\documentclass[10pt]{amsart}
\usepackage{amssymb,amsmath,amsfonts,amsthm}
\usepackage[a4paper]{geometry}
\newtheorem{thm}{Theorem}
\newtheorem{lem}{Lemma}[section]
\newtheorem{cor}[lem]{Corollary}
\theoremstyle{remark}
\newtheorem{expl}{Example}[section]
\newtheorem{rem}[expl]{Remark}
\def\D{\mathrm D}
\def\i{\mathrm i}
\def\d{\mathrm d}
\def\R{\mathbb R}
\def\C{\mathbb C}
\DeclareMathOperator{\diag}{diag}
\DeclareMathOperator*{\slim}{s-lim}
\let\Im\relax
\let\Re\relax
\DeclareMathOperator{\Im}{Im}
\DeclareMathOperator{\Re}{Re}
\title{$C^m$-theory of damped wave equations with stabilisation}
\author{Fumihiko Hirosawa}
\address{F. Hirosawa, Department of Mathematical Sciences, Yamaguchi University, 753-8512, JAPAN}
\email{hirosawa@yamaguchi-u.ac.jp}
\thanks{The first author is supported by Grant-in-Aid for Young Scientists (B) 19740072 from the Ministry of Education, Science, Sports and Culture of Japan.}
\author{Jens Wirth}
\address{J. Wirth, Department of Mathematics, Imperial College, London SW7 2AZ, UK}
\email{j.wirth@imperial.ac.uk}
\thanks{The second author is supported by EPSRC grant EP/E062873/1}
\begin{document}

\begin{abstract}
The aim of this note is to extend the energy decay estimates from [J. Wirth, {\em Wave equations with time-dependent-dissipation I. Non-effective dissipation.}, J. Differential Equations  222 (2006) 487--514] 
to a broader class of time-dependent dissipation including very fast oscillations. This is achieved
using stabilisation conditions on the coefficient in the spirit of [F. Hirosawa, {\em On the asymptotic behavior of the energy for wave equations with time-depending coefficients}, Math. Ann.  339/4 (2007) 819--839].
\end{abstract}

\maketitle

\section{The problem under investigation}
We want to investigate the Cauchy problem 
$$ \square u+2b(t)u_t=0,\qquad u(0,\cdot)=u_1,\quad \D_t u(0,\cdot)=u_2 $$
for a weakly damped wave equation with time-dependent dissipation, as usual $\square = \partial_t^2-\Delta$ denotes the d'Alembertian and $\D=-\i\partial$. For this we apply a partial 
Fourier transform to get the ordinary differential equation
$$ \hat u_{tt}+|\xi|^2\hat u+2b(t)\hat u_t = 0 $$
parameterised by the frequency $\xi$. To formulate a first order system corresponding to this
second order equation we consider $V=(|\xi|\hat u,\D_t \hat u)^T$, such that
$$ \D_t V = \begin{pmatrix} & |\xi| \\ |\xi| & 2\i b(t) \end{pmatrix} V = A(t,\xi) V. $$
We denote its fundamental by $\mathcal E(t,s,\xi)$, i.e.
$$ \D_t \mathcal E(t,s,\xi) = A(t,\xi)\mathcal E(t,s,\xi),\qquad \mathcal E(s,s,\xi)=I\in\mathbb C^{2\times 2}. $$
Our aim is to understand the structure of $\mathcal E(t,s,\xi)$ and its asymptotic behaviour as $t\to\infty$
in dependence on the coefficient function $b=b(t)$ extending results from \cite{Wirth06}.  
If $b(t)\equiv 0$ vanishes identically, we denote the fundamental solution as $\mathcal E_0(t,s,\xi)$ and refer to it as free solution.

There is a very strong interrelation between properties of the coefficient function $b=b(t)$ and decay
properties of solutions to the above Cauchy problem. We refer to \cite{Wirth05}, \cite{Wirth06}, \cite{Wirth:2007} for an overview of related results, which is complete at least for monotonous coefficients and provides sharp
decay results for solutions. However, if the coefficient functions are allowed to bear a certain amount of
oscillations, results may change dramatically. For the case of variable propagation speed this may even lead to exponentially growing energy, as pointed out in \cite{Yagdjian:2001} or, at least destroy the structure of decay results, \cite{Reissig:1999}, \cite{Yagdjian:2000a}, \cite{Reissig:2005}. Recently, the first author developed a technique to obtain
positive results for similar situations with strong oscillations by using a refined diagonalisation technique and a so-called stabilisation condition on coefficients,  \cite{Hirosawa06}. The aim of this note is to extend this technique to the situation of oscillations in lower order terms, especially oscillations in dissipation terms. 

This note is organised as follows. At first we will introduce in Section~\ref{sec2} the assumptions
we impose on the coefficient function $b=b(t)$; particular examples of admissible coefficients are given in Section~\ref{sec3}. The construction of $\mathcal E(t,s,\xi)$ and will be done in Section~\ref{sec4},
where we introduce zones and give precise information on the structure of the fundamental solution
depending on corresponding areas of the phase space. Finally, Section~\ref{sec5} collects the main results of this note. We present two theorems describing sharp energy decay results for solutions of
the above introduced Cauchy problem.

Throughout these notes we denote by $c$ or $C$ various constants which may change from line to line. Furthermore, $f(p)\lesssim g(p)$ for two positive functions means that there exists a constant such that
$f(p)\le Cg(p)$ for all values of the parameter $p$. Similarly, $f(p)\gtrsim g(p)$ means $g(p)\lesssim f(p)$
and $f(p)\approx g(p)$ stands for $f(p)\lesssim g(p)$ and $g(p)\lesssim f(p)$. For a matrix $A$ we denote by $||A||$ its spectral norm, while $|A|$ stands for the matrix composed of the absolute values
of the entries.

\section{Assumptions}\label{sec2}
Tools used in the approach are closely related to conditions on the coefficients. We impose that the dissipation can be written as
$$2b(t)=\mu(t)+\sigma(t),$$
where functions $\mu(t)$ and $\sigma(t)$ carry different kind of information: $\mu(t)$ will determine
the {\em shape} of the coefficient, while $\sigma(t)$ contains {\em oscillations} (and zero mean in a certain sense). In detail our assumptions are
\begin{description}
\item[(1)] $\mu(t)>0$, $\mu'(t)<0$ and $\limsup\limits_{t\to\infty}t\mu(t)<1$;
\item[(2)]  {\em generalised zero mean condition}
$$ \sup_t \left|\int_0^t \sigma(s)\d s\right| <\infty; $$
\item[(3)] {\em stabilisation condition}
$$ \int_0^t \left| \exp\left(\int_0^\theta\sigma(s)\d s\right) -\omega_\infty \right|\d \theta \lesssim \Theta(t)=o(t) $$
with a suitable (uniquely determined) $\omega_\infty>0$ (and a function $\Theta(t)$ normalised such that $\Theta(0)=\mu(0)$);
\item[(4)] {\em symbol-like conditions} for derivatives
$$ \left|\frac{\d^k}{\d t^k} b(t)\right| \le C_k  \Xi(t)^{-k-1},\qquad k=1,2,\dots,m $$
with $\Xi(t) \gtrsim \Theta(t)$;
\item[(5)] together with the {\em compatibility condition} 
$$   \int_t^\infty \Xi(s)^{-m-1}\d s\lesssim \Theta(t)^{-m} $$
between (3) and (4). 
\end{description}

Let's explain the philosophy behind the conditions (1) to (5): We assume that the reader is familiar with
\cite{Wirth06} and/or \cite{Hirosawa06}. Since $\mu(t)$ should describe the shape of the coefficient, the assumptions of (1) are related to \cite{Wirth05}, \cite{Wirth06}. The limit conditions excludes the exceptional case from \cite{Wirth04}, where a structural change in the representation of solutions occurrs. Condition (2) describes that $\sigma(t)$ contains oscillations and the integral assumption implies a zero mean condition of $\sigma(t)$. The stabilisation condition (3) is related to \cite{Hirosawa06} and \cite{HW} (after a Liouville type
transformation of variables). Note that (2) implies $\Theta(t)=\mathcal O(t)$, 
the stabilisation improves this trivial estimate. We will use the notation
$$ f(t) \rightsquigarrow \alpha \qquad\text{if and only if}\qquad \int_0^t |f(s)-\alpha|\d s = o(t) $$
for stabilising functions. Some elementary properties are given later on. The symbol-like estimates
of assumption (4) are thought to be weaker than the ones from \cite{Wirth05}, \cite{Wirth06}, 
where $\Xi(t)=(1+t)$ was used. Stabilisation allows to use weaker assumptions on derivatives by
shrinking the hyperbolic zone to
$$ Z_{hyp}(N) = \{ (t,\xi)\;|\; \Theta(t)|\xi| \ge N\}. $$
We pay for this by using {\em more} steps of diagonalisation. The number of steps for diagonalisation 
will be the number $m$ from condition (5). It implies that remainder terms are uniformly integrable over
the hyperbolic zone after applying $m$ steps.

\begin{rem}
According to Appendix A.6 and under assumption (2) the condition (3) holds if and only if $\int_0^t \sigma(s)\d s\rightsquigarrow \log\omega_\infty$. We will exploit this fact in the examples stated below.
\end{rem}

\section{Examples}\label{sec3}
We will collect some examples to illustrate the nature of our assumptions.

\begin{expl}\label{expl31}
 First we set
$\mu(t)=\frac{\mu}{1+t}$ with a fixed constant $\mu\in(0,1/2)$. Then (1) is fulfilled. 
Furthermore, $\sigma(t) = \mu(t)\sin(t^\alpha)$ satisfies (2) to (5). Indeed, (2) follows from
$$ \int_0^t \frac{\sin(s^\alpha)}{1+s} \d s = \frac 1\alpha \int_0^{t^\alpha} \frac{\sin\theta}{\theta^{(\alpha-1)/\alpha}+\theta}\d\theta \approx \int^{t^\alpha} \frac{\sin\theta}{1+\theta}\d\theta \approx \mathrm{Si}(t^\alpha) = \mathcal O(1), $$
while for (3) we use that the above integral converges by Leibniz criterium such that with 
$\omega_\infty = \exp(\int_0^\infty \frac{\sin(s^\alpha)}{1+s} \d s )>0$ 
\begin{align*}
\int_{t_0}^t& \left| \exp(\int_0^\theta \frac{\sin(s^\alpha)}{1+s} \d s) - \omega_\infty\right|\d \theta
= \omega_\infty \int_{t_0}^t \left|\exp(\int_{\theta}^\infty \frac{\sin(s^\alpha)}{1+s} \d s)-1\right|\d \theta
\\&\lesssim \int_{t_0}^t \left| \int_\theta^\infty \frac{\sin(s^\alpha)}{1+s} \d s\right| \d\theta
\approx \int_{t_0}^t \left| \int_{\theta^\alpha}^\infty \frac{\sin(s)}{1+s} \d s\right| \d\theta \le
\int_{t_0}^t \theta^{-\alpha}\d\theta \\
&\approx t^{1-\alpha}=\Theta(t)=o(t)
\end{align*}
for $\alpha\in(0,1)$.  
Furthermore, $| \d_t^k \sin(t^\alpha)/(1+t) | \le C_k (1+t)^{-1-k(1-\alpha)}$.
Condition (4) and (5) are satisfied for $m=1$ if we take $\Xi(t)=(1+t)^{1-\alpha/2}$.
\end{expl}

\begin{expl}%[Counterexample]
Let $\mu(t)=\frac{\mu}{1+t}$.
If we consider $\sigma(t)=\mu(t)\sin(t/\log(e+t))$ we find (similarly) that (1) to (3) are satisfied
with $\Theta(t)=\log^2(e+t)$. Derivatives satisfy $|\d^k_t \sin(t/\log(e+t))|\le C_k (1+t)^{-1}\log(e+t)^{-k}$.
We can not satisfy (4) and (5). 
\end{expl}

\begin{expl}
For $\sigma(t)=(1+t)^{-\beta}\sin(t^\alpha)$ assumptions (2) and (3) are satsified if $\alpha,\beta>0$ and $1<\alpha+\beta<2$. We get $\Theta(t)=t^{2-\alpha-\beta}$ and assumptions (4) and (5) can hold only if $\beta\ge1$. Combining this with $\mu(t)=\frac{\mu}{1+t}$, $\mu<1/2$, we can choose $m=1$ and $\Xi(t)=t^{(1+\beta)/2-\alpha/2}$.
\end{expl}

\begin{expl}\label{expl34}
If we consider $\mu(t) = \frac1{(1+t)\log(e+t)}$, then the function $\sigma(t)=\mu(t)\sin(t/\log(e+t))$ 
may be chosen. In this case $\Theta(t)=\log(e+t)$. For (4) and (5) we choose $m=1$ and 
 $\Xi(t)=\sqrt{1+t}\log(e+t)$,
$$ \int_t^\infty \frac1{s\log^2s}\d s\approx \frac1{\log t}. $$
\end{expl}

\begin{expl}\label{expl35}
We give a further example of  a coefficient function with higher $m$.  Let $\chi\in C_0^\infty(\R)$
with $\mathrm{supp}\,\chi=[-1,1]$, $|\chi(t)|<1$ and $\int_{-1}^1 \chi(t)\d t=0$. For given sequences $t_j\ge1 $ and $\delta_j\le 1$, $j=1,2,\dots$, of positive real numbers with $t_j+\delta_j^{-1}\le t_{j+1}-\delta_{j+1}^{-1}$ and a suitably chosen real number $\gamma>0$ we define
$$
\sigma(t) = \begin{cases} t_j^{-\gamma} \chi(\delta_j (t-t_j)),\qquad &t\in I_j=[t_j-\delta_j^{-1},t_j+\delta_j^{-1}],\\ 0, &t\not\in \bigcup_j I_j.\end{cases}
$$
Then condition (2) holds because $|\int_0^t\sigma(s)\d s| \le 2$. For condition (3)  we set $\omega_\infty = 1$ and estimate the integral as
$$ \int_0^t \left| \int_0^\tau\sigma(s) \d s \right| \d \tau \le \sum_{j=1}^n 2 = 2n $$
for $t\in[t_n-\delta_n^{-1},t_{n+1}-\delta_{n+1}^{-1}]$.
Thus with $t_n=n^\alpha$ assumption (3) holds for $2n = 2t_n^{1/\alpha}\approx t^{1/\alpha}=\Theta(t)=o(t)$,
i.e. if $\alpha>1$. We set $\delta_n=t_n^{-\gamma}=n^{-\alpha\gamma}$. Then $n^{\alpha}+n^{\alpha\gamma} \le  (n+1)^\alpha-(n+1)^{\alpha\gamma}$ holds true for large $n$ if $\gamma<1$.
Derivatives satisfy $|\d_t^k\mu(t)\sigma(t)|\lesssim  (1+t_j)^{-\gamma(k+1)} \approx (1+t)^{-\gamma(k+1)}$, thus (4) holds with $\Xi(t) = (1+t)^{\gamma}$. Now we choose $\gamma$ such that (5)  
holds for a given number $m$, thus 
$$ \int_t^\infty (1+\tau)^{-(m+1)\gamma} \d\tau \lesssim t^{-m/\alpha},$$
i.e. $(m+1)\gamma>1$ and $(m+1)\gamma-1=m/\alpha$. Hence we choose 
$\gamma=\frac1{m+1}+\frac m{\alpha(m+1)}$, which is smaller than 1 for all $m=1,2,\dots$. 
The function $\mu(t)$ may be chosen as in Example~\ref{expl31} or~\ref{expl34}. 

This example shows that for any given number $m$ and any stabilisation rate $\Theta(t)=t^{1/\alpha}$,
$\alpha>1$, we find a coefficient $2b(t)=\mu(t)+\sigma(t)$ subject to (2)--(5).
\end{expl}

\begin{rem}
The results derived later on will show that essential influence on decay properties of solutions comes from the shape function $\mu(t)$, while the `perturbation' $\sigma(t)$ has no influence on the decay rate
at all. Note, that we do {\em not} require that $\sigma(t)$ is small corresponding to $\mu(t)$ in some $L^\infty$ sense. We only require, that the oscillations contained in $\sigma(t)$ are 'neatly arranged'.
\end{rem}

\section{Construction of the fundamental solution $\mathcal E(t,s,\xi)$}\label{sec4}
Main point of our concern is how to use assumption (3) for small frequencies. We will introduce a bit more of notation and denote $\mathcal E_\mu(t,s,\xi)$ the fundamental solution for the case 
$\sigma(t)\equiv0$, 
$$\D_t\mathcal E_\mu(t,s,\xi)=A_\mu(t,\xi)\mathcal E_\mu(t,s,\xi),\qquad \mathcal E_\mu(s,s,\xi)=I.$$ Properties of 
$\mathcal E_\mu(t,s,\xi)$ are studied in \cite{Wirth05}, \cite{Wirth06} (as low-regularity theory, i.e. without using symbol classes and further steps of diagonalisation). The basic behaviour of $\mathcal E_\mu(t,0,\xi)$ can be summarised as follows: The phase space can be decomposed into three parts (two zones
and subzones), 
\begin{itemize}
\item the {\em dissipative zone} $Z_{diss}^{(\mu)}(N) = \{ (t,\xi)\;|\; |\xi| \le N \mu(t)\}$ and
\item the {\em hyperbolic zone}  $Z_{hyp}^{(\mu)}(N) = \{ (t,\xi)\;|\; |\xi|\ge N\mu(t)\}$, divided into
the regions where $|\xi|\ge N\mu(0)$ and where $N\mu(t)\le |\xi|\le N\mu(0)$.
\end{itemize}
The latter subdivision is merely for convenience and does not stand for any deep structural differences
of the fundamental solution. The subdivision into zones is {\em essential} as the following results show.
The constant $N$ does not matter in this case. In the hyperbolic zone the fundamental solution behaves
like $\mathcal E_0(t,s,\xi)$ multiplied by $\lambda(s)/\lambda(t)$ and a uniformly bounded and invertible matrix. The function 
$$\lambda(t)=\exp\left(\frac12\int_0^t \mu(s)\d s\right) $$
contains (the essence of) the influence of dissipation. In contrast to that the fundamental solution $\mathcal E_\mu(t,0,\xi)$  behaves in the dissipative zone essentially like (see e.g.
\cite{Emmerling07} for a neat argument)
$$\diag\left(1,\frac{\lambda^2(0)}{\lambda^2(t)}\right).$$
This {\em bad} behaviour (bad in the sense that it destroys the energy estimate appearing naturally within the hyperbolic zone) has to be compensated by assumptions on the data. Our assumption (1)
on $\mu(t)$ implies
$$ \left\|\mathcal E_\mu(t,0,\xi) \diag(|\xi|/\langle\xi\rangle,1) \right\| \lesssim \frac1{\lambda(t)} $$
with $\langle\xi\rangle = \sqrt{1+|\xi|^2}$.

\subsection{Estimates in the dissipative zone} 
Following the argumentation of \cite{Emmerling07} or
\cite{Wirth06} we see that the fundamental solution $\mathcal E(t,0,\xi)$ satisfies within the 
dissipative zone $$Z_{diss}(N)=\{(t,\xi)\;|\; |\xi|\le N\mu(t)\}=Z_{diss}^{(\mu)}(N)$$ the same estimates
as sketched above for $\mathcal E_\mu(t,0,\xi)$. For later use we will denote the boundary
of the dissipative zone by $t_\xi^{(1)}$.

\begin{lem}\label{lem41}
The fundamental solution $\mathcal E(t,0,\xi)$ satisfies for all
$(t,\xi)\in Z_{diss}(N)$ the point-wise estimate 
$$
   |\mathcal E(t,0,\xi)|\lesssim \frac1{\lambda^2(t)} \begin{pmatrix}|\xi|^{-1} & 1 \\ |\xi|^{-1} & 1 \end{pmatrix}.
$$
\end{lem}
\begin{proof} The proof is mainly taken from \cite[Lemma 2.1]{Wirth:2007a}. We rewrite the system as
system of integral equations, denoting the entries of the rows of $\mathcal E(t,0,\xi)$ as $v(t,\xi)$ and $w(t,\xi)$. This gives
\begin{align*}
  v(t,\xi)&=\eta_1+\i|\xi|\int_0^t w(\tau,\xi)\d\tau,\\
  w(t,\xi)&=\frac1{\tilde\lambda^2(t)}\eta_2 +\i |\xi|\frac1{\tilde\lambda^2(t)}\int_0^t
  \tilde\lambda^2(\tau)v(\tau,\xi)\d\tau,
\end{align*} 
where $\tilde\lambda(t)=\exp(\int_0^t b(\tau)\d\tau)$ and $\eta=(\eta_1,\eta_2)=(1,0)$
or $\eta=(0,1)$ for the first and second column, respectively. 

We start by considering the first column. Plugging the second equation into the first one and interchanging the order of integration gives
\begin{align*}
 v(t,\xi) &= 1- |\xi|^2\int_0^t \frac1{\tilde\lambda^2(\tau)} \int_0^\tau  \tilde\lambda^2(\theta) v(\theta,\xi)\d \theta \d\tau\\
 &= 1- |\xi|^2\int_0^t \tilde\lambda^2(\theta) v(\theta,\xi)\int_\theta^t \frac{\d\tau}{\tilde\lambda^2(\tau)}\d\theta, 
 \end{align*}
such that $\tilde\lambda^2(t)|\xi|v(t,\xi)$ satisfies an Volterra integral equation
$$ |\xi|\tilde\lambda^2(t)v(t,\xi) = h(t,\xi) + \int_0^t k(t,\theta,\xi) |\xi|\tilde\lambda^2(\theta)v(\theta,\xi) \d\theta$$
 with kernel $k(t,\theta,\xi) = 
-|\xi|^2\tilde\lambda^2(t)\int_\theta^t \d\tau/\tilde\lambda^2(\tau)$ and source term $h(t,\xi)=|\xi|\tilde\lambda^2(t)$. Assumptions (1) and (2) imply $h(t,\xi)\lesssim 1$ uniformly on $Z_{diss}(N)$. Representing the solution as Neumann series,
$$
\tilde\lambda^2(t)|\xi|v(t,\xi) = h(t,\xi)+\sum_{\ell=1}^\infty \int_0^t k(t,t_1,\xi)\cdots\int_0^{t_{\ell-1}} k(t_{\ell-1},t_\ell,\xi) h(t_\ell,\xi)\d t_\ell\cdots\d t_1,
$$
implies the bound $\tilde\lambda^2(t)|\xi|v(t,\xi) \in L^\infty(Z_{diss}(N))$ following from the kernel estimate
\begin{align*}
\sup_{(t,\xi)\in Z_{diss}(N)}& \int_0^t \sup_{0\le \tilde t\le t_\xi^{(1)}} |k(\tilde t,\theta,\xi)|\d\theta
\lesssim \sup_{(t,\xi)\in Z_{diss}(N)} |\xi|^2 \lambda^2(t_\xi^{(1)}) \int_0^{t_\xi^{(1)}} \int_\theta^{t_\xi^{(1)}}
\frac{\d\tau}{\lambda^2(\tau)}\\ 
\\&= \sup_{(t,\xi)\in Z_{diss}(N)} |\xi|^2\lambda^2(t_\xi^{(1)}) \int_0^{t_\xi^{(1)}} \frac{\tau}{\lambda^2(\tau)}\d\tau \lesssim  (|\xi| t_\xi^{(1)}) ^2\lesssim1,
\end{align*}
based on $|k(t,\theta,\xi)|\approx |\xi|^2\lambda^2(t)\int_\theta^t \d\tau/\lambda^2(\tau)$ (by (2))
and the monotonicity of $t/\lambda^2(t)$ for large time. The second integral equation implies the
corresponding bound for $w(t,\xi)$,
$$
|\xi|\tilde\lambda^2(t) |w(t,\xi)|\le |\xi| \int_0^t|\xi|
  \tilde\lambda^2(\tau)|v(\tau,\xi)|\d\tau\lesssim |\xi| \int_0^t\d\tau\lesssim 1
$$
uniformly on $Z_{diss}(N)$.

For the second column we use the same idea: Plugging the second equation into the first one yields the new source term $|\xi|\int_0^t \d\tau/\tilde\lambda^2(\tau)\lesssim 1/\lambda^2(t)$. Therefore the representation as Neumann series yields $\lambda^2(t)v(t,\xi)\in L^\infty(Z_{diss}(N))$ and integration with the second integral equation gives consequently $\lambda^2(t)w(t,\xi)\in L^\infty(Z_{diss}(N))$.
The statement is proven.
\end{proof}
\begin{cor}\label{cor41a}
The fundamental solution $\mathcal E(t,0,\xi)$ satisfies within the dissipative zone 
$Z_{diss}(N)$ the norm-estimate
$$ \left\|\mathcal E(t,0,\xi) \diag(|\xi|/\langle\xi\rangle,1) \right\| \lesssim \frac1{\lambda^2(t)}. $$
\end{cor}
\begin{proof}
  It remains to estimate the entries in the first column of the product on the right. Lemma~\ref{lem41}
  gives the uniform bound $(\lambda^2(t)|\xi||)^{-1} |\xi|/\langle\xi\rangle \lesssim 1/\lambda^2(t)$.
\end{proof}

\begin{rem} 
Later on we will only use the estimate $1/\lambda(t)$ within this zone. 
\end{rem}

\subsection{Diagonalisation -- Estimates in the hyperbolic zone}
We sketch the results, the idea of proof is essentially the same as in \cite{Hirosawa06} or \cite{Wirth06}. A Fourier multiplier $a(t,\xi)$ belongs to the symbol class $\mathcal S_{N}^{\ell}\{m_1,m_2\}$ if it satisfies the symbol estimate
$$ \left|\D_t^k\D_\xi^\alpha a(t,\xi)\right| \le C_{k,\alpha}  |\xi|^{m_1-|\alpha|} \Xi(t)^{-m_2-k} $$
for all $(t,\xi)\in Z_{hyp}(N)$ and $k\le\ell$ and all multi-indices $\alpha\in\mathbb N_0^n$. These symbol classes satisfy natural calculus rules (like the ones from \cite[Proposition 6]{Wirth06}). The condition $\Xi(t)\gtrsim\Theta(t)$ gives embedding relations for these symbol classes of the form
$$  \mathcal S_{N}^{\ell}\{m_1,m_2\} \hookrightarrow \mathcal S_{N}^{\ell}\{m_1+k,m_2-k\} ,\qquad k\ge 0,$$
which is important for the application of the diagonalisation scheme sketched below.
Assumption (4) guarantees that $b(t)\in\mathcal S_{N}^{m} \{0,1\}$, such that for $(t,\xi)\in Z_{hyp}(N)$ the matrix $A(t,\xi)$ consists of a main part from $\mathcal S_{N}^{\infty}\{1,0\}$
 and the lower order corner entry from $\mathcal S_{N}^{m} \{0,1\}\hookrightarrow
 \mathcal S_{N}^{m}\{1,0\}$ . 

The following diagonalisation scheme is merely standard and adapted from \cite{Hirosawa06}. In an introductory step we diagonalise the homogeneous main part using $M=(\begin{smallmatrix} 1&-1\\1&1\end{smallmatrix})$, such that $V^{(0)}=M^{-1}V$ satisfies for all $(t,\xi)\in Z_{hyp}(N)$
$$ \D_t V^{(0)} = (\mathcal D_0(t,\xi)+R_0(t,\xi)) V^{(0)} $$
with $\mathcal D_0(t,\xi)=\diag(|\xi|+\i b(t),-|\xi|+\i b(t))$ and $R_0(t,\xi)={\i b(t)}(\begin{smallmatrix}0&1\\1&0\end{smallmatrix})$. This system is diagonal modulo $R_0\in\mathcal S_{N}^{m}\{0,1\}$. Now we apply an iterative procedure to diagonalise it modulo 
$\mathcal S_{N}^{m-k}\{-k,k+1\}$, $k=1,\dots,m$.

\begin{lem}
There exists a zone constant $N>0$ such that for any $k=0,1,\ldots m$ there exist matrices 
\begin{itemize}
\item $N_k(t,\xi)\in\mathcal S_N^{m-k}\{0,0\}$, invertible with inverse $N_k^{-1}\in\mathcal S_N^{m-k}\{0,0\}$
and tending to the identity as $t\to\infty$ for all fixed $\xi\ne0$;
\item $R_k(t,\xi)\in\mathcal S_N^{m-k}\{-k,k+1\}$;
\item $\mathcal D_k(t,\xi)\in\mathcal S_N^{m-k}\{1,0\}$ diagonal, $\mathcal D_k(t,\xi)=\diag(\tau_k^+(t,\xi),\tau_k^{-1}(t,\xi))$
\end{itemize}
satisfying the operator identities
$$ (\D_t-\mathcal D_k-R_k)N_k=N_k(\D_t-\mathcal D_{k+1}-R_{k+1}) $$
for $k=0,1,\ldots,m-1$.
\end{lem}

\begin{proof} The proof goes by direct construction.
Assume for this that we have given a system $\D_t V^{(k)} = (\mathcal D_k(t,\xi)+R_k(t,\xi)) V^{(k)}$
with $\mathcal D_k(t,\xi) = \diag(\tau_k^+(t,\xi),\tau_k^-(t,\xi))\in\mathcal S_N^{m-k}\{1,0\}$
satisfying 
$$ |\delta_k(t,\xi)| = |\tau_k^+(t,\xi)-\tau_k^-(t,\xi)| \ge C_k |\xi| $$
and antidiagonal $R_k(t,\xi)  \in \mathcal S_{N}^{m-k}\{-k,k+1\}$. Then we denote the difference of the diagonal entries as $\delta_k(t,\xi)=\tau_k^+(t,\xi)-\tau_k^-(t,\xi)$ and set 
$$ N_k(t,\xi) = I + \begin{pmatrix} 0 & -(R_k)_{12} / \delta_k \\ (R_k)_{21} / \delta_k & 0\end{pmatrix} $$
such that $[\mathcal D_k, N_k] = -R_k$ and therefore
\begin{align*}
 B^{(k+1)}&=(\D_t - \mathcal D_k - R_k) N_k - N_k (\D_t - \mathcal D_{k})\\
&= \D_t N_k - [\mathcal D_k,N_k]  - R_k N_k\\
&= (\D_t N_k) - R_k (N_k-I) \in   \mathcal S_{N}^{m-k-1,\infty}\{-k-1,k+2\}.
\end{align*}
The matrix $N_k(t,\xi)$ is invertible, if we choose the zone constant $N$ sufficiently large. This
follows from the symbol estimate $I-N_k \in  \mathcal S_{N}^{m-k}\{-k-1,k+1\}$. Thus by defining
$ \mathcal D_{k+1} = \mathcal D_k-\diag(N_k^{-1} B^{(k+1)}) $ and 
$R_{k+1}=\diag(N_k^{-1} B^{(k+1)}) -N_k^{-1} B^{(k+1)}$ we obtain the operator equation
$$(\D_t-\mathcal D_k-R_k)N_k = N_k (\D_t-\mathcal D_{k+1}-R_{k+1})$$
and it is easily checked that the assumptions we made are satisfied again. 
\end{proof}

Finally we obtain for
$k=m$ that the remainder $R_m(t,\xi)\in\mathcal S_{N}^{0} \{-m,m+1\}$ is uniformly integrable over
the hyperbolic zone,
$$ \int_{t_\xi^{(2)}}^\infty ||R_m(t,\xi)||\d t \le |\xi|^{-m} \int_{t_\xi^{(2)}}^\infty \Xi(t)^{-m-1} \d t
\le |\xi|^{-m} \Theta(t_\xi^{(2)})^{-m}\le N,$$ 
where we defined $t_\xi^{(2)}$ as the maximum of $0$ and the implicitly defined zone boundary $\Theta(t_\xi^{(2)})|\xi|=N$.  To complete the construction of our representation we need more information on the diagonal matrices $\mathcal D_k$.

\begin{lem}
For all $k=0,1,\ldots, m$ the difference of the diagonal
entries of $\mathcal D_k(t,\xi)$ is real.
\end{lem}

\begin{proof} 
Again we proceed by induction over $k$ and follow the diagonalisation scheme. For $k=0$ the assertion is satisfied and the hypothesis 

(H) $R_k(t,\xi)$ has the form $R_k = \i\big(\begin{smallmatrix} & \overline\beta_k \\ \beta_k &\end{smallmatrix}\big)$ with complex-valued $\beta_k(t,\xi)$

is true. Thus the construction implies $N_k = I  +  \frac\i{\delta_k} \big(\begin{smallmatrix} & -\overline\beta_k\\\beta_k \end{smallmatrix}\big)$ with $\det N_k =1- |\beta_k|^2/\delta_k^2 \ne0$ (for our choice of the zone constant $N$).  Following \cite{Hirosawa06} we obtain (with $d_k = |\beta_k|^2/\delta_k$)
$$ N_k^{-1} (\mathcal D_k+R_k) N_k 
= \frac1{1-d_k} \big( \diag\big(\tau_k^+ - d_k\tau_k^+ - \delta_kd_k, \tau_k^--d_k\tau_k^-+\delta_kd_k\big) +d_kR_k\big) $$
and
$$ N_k^{-1}(\D_t N_k) = \frac1{1-d_k} \left(\begin{pmatrix} \i\frac{\overline\beta_k}{\delta_k} \partial_t \frac{\beta_k}{\delta_k} &   \\ &
\i \frac{\beta_k}{\delta_k} \partial_t \frac{\overline\beta_k}{\delta_k} 
\end{pmatrix} + \begin{pmatrix}& -\partial_t \frac{\overline\beta_k}{\delta_k}\\\partial_t \frac{\beta_k}{\delta_k} \end{pmatrix}\right)
$$
such that $\Re \frac{\beta_k}{\delta_k} \partial_t \frac{\overline\beta_k}{\delta_k} =\frac{ \partial_td_k}2 = \Re 
\frac{\overline\beta_k}{\delta_k} \partial_t \frac{\beta_k}{\delta_k} $ implies
$$ \tau_{k+1}^\pm = \tau_k^\pm \mp \frac1{1-d_k} \left(d_k\delta_k +\Im \left( \frac{\beta_k}{\delta_k} \partial_t \frac{\overline\beta_k}{\delta_k} 
\right)\right) - \i\frac{\partial_t d_k}{2(d_k-1)}.$$
Hence $\delta_{k+1}$ is real again and $R_{k+1}$ satisfies (H) and, therefore, both statements are true for all $k$ up to $m$. 
\end{proof}

Now the construction of the fundamental solution $\mathcal E(t,s,\xi)$ is merely standard. At first we solve the diagonal system $\D_t-\mathcal D_m(t,\xi)$. Its fundamental solution is given by 
$$\exp\left(\i\int_s^t \mathcal D_m(\theta,\xi)\d\theta\right) = \diag\left( e^{\i\int_s^t\tau_m^+(\theta,\xi)\d\theta},  e^{\i\int_s^t\tau_m^-(\theta,\xi)\d\theta}\right) .$$ 
Since $\delta_m=\tau_m^+-\tau_m^-$ is real, it follows that $\Im\tau_m^+ = \Im\tau_m^-=:\Im\tau_m$ and
thus the matrix
$$ \exp\left(\int_s^t \Im\tau_m(\theta,\xi)\d\theta\right) \,\exp\left(\i \int_s^t \mathcal D_m(\theta,\xi)\d\theta\right) $$
is unitary. Note, that the first factor is scalar. Now the integrability of the remainder term $R_m(t,\xi)$ 
over the hyperbolic zone implies (like in \cite{Wirth06}) that the fundamental matrix of
$\D_t-\mathcal D_m-R_m$ is given by $\exp(\i\int_s^t \mathcal D_m(\theta,\xi)\d\theta) \mathcal Q_m(t,s,\xi)$
with a uniformly bounded and invertible matrix $\mathcal Q_m(t,s,\xi)$,
$$
\mathcal Q_m(t,s,\xi) = I+\sum_{k=1}^\infty \int_s^t \tilde R_m(t_1,s,\xi)\int_s^{t_1}\tilde R_m(t_2,s,\xi)\cdots\int_s^{t_{k-1}} \tilde R_m(t_k,s,\xi)\d t_k\cdots\d t_2\d t_1
$$
where 
$$\tilde R_m(t,s,\xi)=\exp(-\i\int_s^t \mathcal D_m(\theta,\xi)\d\theta)
R_m(t,\xi)\exp(\i\int_s^t\mathcal D_m(\theta,\xi)\d\theta)$$ is an auxiliary function. The representation implies that $\mathcal Q_m(t,s,\xi)$ tends to the identity as $t\to\infty$
locally uniform in $s$ and $\xi$. Collecting these results we obtain

\begin{lem}\label{lem44}
The fundamental matrix $\mathcal E(t,s,\xi)$ satisfies 
$$ \mathcal E(t,s,\xi) = M^{-1} \left(\prod_{k=0}^{m-1} N_k^{-1}(t,\xi)\right)
\exp\left(\i\int_s^t \mathcal D_m(\theta,\xi)\d\theta\right) \mathcal Q_m(t,s,\xi)
 \left(\prod_{k=0}^{m-1} N_k(s,\xi)\right) M $$
for all $(t,\xi),(s,\xi)\in Z_{hyp}(N)$, where
\begin{itemize}
\item the matrices $N_k(t,\xi)$ are uniformly bounded
and invertible with $N_k(t,\xi)\to I$ and
\item $\mathcal Q_m(t,s,\xi)$ is uniformly bounded satisfying
$\mathcal Q_m(t,s,\xi)\to\mathcal Q_m(\infty,s,\xi)$, 
\end{itemize}
both limits locally uniform in $\xi$ as $t\to\infty$.
\end{lem} 

Hence,
the time-asymptotics of solutions is encoded in the real-valued function $\Im\tau_m(t,\xi)$,
$$ ||\mathcal E(t,s,\xi)|| \approx \exp\left(-\int_s^t \Im\tau_m(\theta,\xi)\d\theta\right), \qquad t\to\infty $$
locally uniform in $\xi$ for fixed $s$  (and such that $(s,\xi)\in Z_{hyp}(N)$). We can use our representation of $\tau_m(t,\xi)$ to deduce
$$ \Im\tau_m(t,\xi) = b(t) + \sum_{j=1}^{m-1} \frac {\partial_t d_k }{2(d_k-1)} $$
such that
$$ \exp\left(-\int_s^t \Im \tau_m(\theta,\xi)\d\theta\right) = \exp\left(-\int_s^t b(\theta)\d\theta\right)
\prod_{j=1}^{m-1} \left(\frac{d_k(t,\xi)-1}{d_k(s,\xi)-1}\right)^{-1/2} \approx \frac{\lambda(s)}{\lambda(t)}.$$

\begin{cor}\label{cor45}
The fundamental matrix $\mathcal E(t,s,\xi)$ satisfies uniformly in $(t,\xi),(s,\xi)\in Z_{hyp}(N)$
the two-sided estimate
$$ ||\mathcal E(t,s,\xi)||\approx\frac{\lambda(s)}{\lambda(t)}.$$
\end{cor}

\subsection{Estimates in the intermediate zone} 
Since $\Theta(t) \mu(t) \lesssim \Theta(t) / (1+t) = o(1)$, there remains a gap between the dissipative and the hyperbolic zone. We will denote this zone as 
$$Z_{int}(N)=\{(t,\xi)\,|\, t_\xi^{(1)}\le t\le t_\xi^{(2)}\}.$$
 In this zone we relate $\mathcal E(t,s,\xi)$ to (known estimates for) $\mathcal E_\mu(t,s,\xi)$ and use the stabilisation condition (3). For this we solve 
 $$
 \D_t \Lambda(t,s,\xi) = \big(A(t,\xi)-A_\mu(t,\xi)\big)\Lambda(t,s,\xi),\qquad \Lambda(s,s,\xi)=I,
 $$
 which gives 
$$ \Lambda(t,s,\xi) = \diag\left(1,\exp\big(-\int_s^t \sigma(\theta)\d\theta\big)  \right),$$
and make the {\em ansatz} $\mathcal E(t,s,\xi)=\Lambda(t,s,\xi)\mathcal R(t,s,\xi)$.
It follows that the matrix $\mathcal R(t,s,\xi)$ satisfies
$$ \D_t\mathcal R(t,s,\xi) = \Lambda(s,t,\xi)A_\mu(t,\xi)\Lambda(t,s,\xi) \mathcal R(t,s,\xi),\qquad \mathcal R(s,s,\xi)=I, $$
where the coefficient matrix in this system has the form
\begin{align*}
  \tilde A_\mu(t,s,\xi) &=\Lambda(s,t,\xi)A_\mu(t,\xi)\Lambda(t,s,\xi)\\&= \begin{pmatrix} 0 & \exp\big(-\int_s^t \sigma(\theta)\d\theta\big) |\xi|  \\ \exp\big(\int_s^t \sigma(\theta)\d\theta\big) |\xi| &\i\mu(t) \end{pmatrix}  
\end{align*}
Note that condition (3) means $\Lambda(t,s,\xi)\rightsquigarrow\diag(1,\hat \omega_\infty(s)^{-1})$, where we use $\hat\omega_\infty(s)=\omega_\infty \exp(-\int_0^s \sigma(\theta)\d\theta)$.
Condition (2) implies that $0<c\le\hat\omega_\infty(s)\le C<\infty$ with suitable constants.
Thus the new speed of propagation satisfies the stabilisation condition (as used in \cite{Hirosawa06}), while the dissipation term has no bad influence as consequence of assumption (1).  

We denote by $\hat A_\mu(t,s,\xi)$ the matrix 
$$\hat A_\mu(t,s,\xi)=\begin{pmatrix}0&\hat\omega_\infty(s)^{-1} |\xi|\\\hat\omega_\infty(s)|\xi|&\i\mu(t)\end{pmatrix} $$ 
and solve the corresponding system $\D_t-\hat A_\mu$. The diagonaliser of the $|\xi|$-homogeneous part is given by
$$ \hat M(s) = \begin{pmatrix} 1 & -1 \\ \hat\omega_\infty(s) & \hat\omega_\infty(s) \end{pmatrix},
\qquad \hat M^{-1}(s) = \frac12  \begin{pmatrix}1 &\hat \omega_\infty(s)^{-1} \\ -1 &\hat\omega_\infty(s)^{-1} \end{pmatrix},$$
such that
$$ \hat M^{-1}(s) \hat A_\mu(t,s,\xi) \hat M(s) = \begin{pmatrix}|\xi|&0\\0&-|\xi|\end{pmatrix} + \frac{\i\mu(t)}2\begin{pmatrix}1&1\\1&1\end{pmatrix}.$$
Surprisingly, this means
$$ M\hat M^{-1}(s) \hat A_\mu(t,s,\xi) \hat M(s) M^{-1} = A_\mu(t,\xi), $$
such that the solution $\hat{\mathcal E}_\mu(t,s,\xi)$ to the auxiliary problem 
$$(\D_t-\hat A_\mu(t,s,\xi))\hat{\mathcal E}_\mu(t,s,\xi)=0,\qquad \hat{\mathcal E}_\mu(s,s,\xi)=I,$$ satisfies 
$\hat{\mathcal E}_\mu(t,s,\xi)=\hat M(s) M^{-1}\mathcal E_\mu(t,s,\xi) M \hat M^{-1}(s)$.  
This relation implies:

\begin{lem}\label{lem46}
The matrix $\hat{\mathcal E}_\mu(t,s,\xi)$ satisfies uniformly in $(t,\xi),(s,\xi)\in Z_{hyp}^{(\mu)}(N)$
$$\|\hat{\mathcal E}_\mu(t,s,\xi)\|\approx \frac{ \lambda(s)}{\lambda(t)}.$$
\end{lem}

Now we use the stabilisation property of $\tilde A_\mu(t,s,\xi)$ to find $\mathcal R(t,s,\xi)$ of the form $\mathcal R(t,s,\xi) = \hat{\mathcal E}_\mu(t,s,\xi) \mathcal Q_{\mathcal R}(t,s,\xi)$.  The coefficient matrix of the differential equation satisfied by $\mathcal Q_{\mathcal R}$,
$$ \D_t\mathcal Q_{\mathcal R}(t,s,\xi)  = \hat{\mathcal E}_\mu(s,t,\xi)\big(\tilde A_\mu(t,s,\xi)-\hat A_\mu(t,s,\xi)\big)\hat{\mathcal E}_\mu(t,s,\xi) \mathcal Q_{\mathcal R}(t,s,\xi), \qquad \mathcal Q_{\mathcal R}(s,s,\xi)  =I,  $$
satisfies the estimate (note, that the two-sided estimates for $\hat{\mathcal E}_\mu$ cancel each other)
\begin{align*}
 \int_s^t &  \|\hat{\mathcal E}_\mu(s,\tau,\xi)\big(\tilde A_\mu(\tau,s,\xi)-\hat A_\mu(\tau,\xi)\big)\hat{\mathcal E}_\mu(\tau,s,\xi)\|\d\tau \\ 
 & \approx |\xi| \int_s^t \left| \exp\left(\int_s^\theta\sigma(\tau)\d \tau\right) -\omega_\infty(s) \right|\d \theta \\
 & \approx |\xi| \int_s^t \left| \exp\left(\int_0^\theta\sigma(\tau)\d \tau\right) -\omega_\infty \right|\d \theta   \le |\xi|\Theta(t) \le N.
\end{align*}
Now the standard construction of $\mathcal Q_{\mathcal R}(t,s,\xi)$ in terms of the Peano-Baker series
(as done for $\mathcal Q_m(t,s,\xi)$ before) gives uniform bounds for this matrix and for its inverse within the intermediate zone. Thus

\begin{lem}
   The fundamental matrix $\mathcal E(t,s,\xi)$ can be represented in $Z_{int}(N)$ as
   $$ \mathcal E(t,s,\xi) = \Lambda(t,s,\xi) \hat{\mathcal E}_\mu(t,s,\xi) \mathcal Q_{\mathcal R}(t,s,\xi), $$
   where $\Lambda(t,s,\xi)$ and $\mathcal Q_{\mathcal R}(t,s,\xi)$ are uniformly bounded in
   $(t,\xi),(s,\xi)\in Z_{int}(N)$ and $\hat{\mathcal E}_\mu(t,s,\xi)$ satisfies the bound of 
   Lemma~\ref{lem46}.
\end{lem}
\begin{cor}\label{cor48}
The fundamental matrix $\mathcal E(t,s,\xi)$ satisfies uniformly in $(t,\xi),(s,\xi)\in Z_{int}(N)$
the two-sided estimate
$$ ||\mathcal E(t,s,\xi)||\approx\frac{\lambda(s)}{\lambda(t)}.$$
\end{cor}

\section{Main results -- Energy decay estimates}\label{sec5}
The results of the previous section can be collected as energy estimates for solutions
to the original Cauchy problem. We obtain two results. The first one states the energy decay
estimate. The use of the $H^1$-norm for the first datum is {\em essential} for the validity of the 
statement. This cancels the $|\xi|^{-1}$ in the estimate of Lemma~\ref{lem41}.

\begin{thm}
Solutions to the Cauchy problem
$$ \square u+2b(t)u_t =0,\qquad u(0,\cdot)=u_1,\quad u_t(0,\cdot)=u_2 $$
for data $u_1\in H^1(\R^n)$ and $u_2\in L^2(\R^n)$ and coefficient function $b(t)$ subject to conditions (1)--(5) satisfy the {\em a-priori} estimate
$$ \|\nabla u(t,\cdot)\|_{L^2} + \|u_t(t,\cdot)\|_{L^2} \le C \frac1{\lambda(t)} \big( \|u_1\|_{H^1}+\|u_2\|_{L^2}\big) $$
with a constant $C$ depending on the size of the coefficient $b(t)$ and its first $m$ derivatives, where  the function $\lambda(t)$ is given by $\lambda(t) = \exp(1/2\int_0^t \mu(\tau)\d\tau)\approx\exp(\int_0^t b(s)\d s$. 
\end{thm}
\begin{proof}
By Plancherel's theorem it is equivalent to prove the corresponding statement in Fourier space,
$$ \|\xi \hat u(t,\cdot)\|_{L^2} + \|\hat u_t(t,\cdot)\|_{L^2} \le C \frac1{\lambda(t)} \big( \|\langle\xi\rangle \hat u_1\|_{L^2}+\|\hat u_2\|_{L^2}\big), $$
which reduces by the unitarity of Riesz transform, $\|\xi\hat u\|_{L^2} = \||\xi|\hat u\|_{L^2}$,  and in view of our system reformulation to the estimate
$$
     \| \mathcal E(t,0,\xi) \diag(|\xi|/\langle\xi\rangle,1)V_0 \| \lesssim \frac1{\lambda(t)} \|V_0\|
$$
for all $V_0\in \C^2$ and uniform in $\xi$. But this is just the combination of Corollaries~\ref{cor41a}, \ref{cor45} and~\ref{cor48}.
\end{proof}

The second result is an application of Banach-Steinhaus theorem on the (dense) subspace
of data for which $0$ does not belong to the Fourier support. It follows essentially from the fact
that the matrix $\mathcal Q(t,s,\xi)$ tends locally uniform to a  invertible matrix $\mathcal Q(\infty,s,\xi)$
inside the hyperbolic zone $Z_{hyp}(N)$. The proof is analogous to the corresponding one in \cite[Theorem 31]{Wirth06}, \cite[Corollary 3.2]{Wirth:2007a}.

\begin{thm}
For any fixed choice of data $u_1\in H^1(\R^n)$ and $u_2\in L^2(\R^n)$
we find constants $c$ and $C$ such that under the assumptions of Theorem~1 the solution to the
Cauchy problem satisfies
$$ c \le \lambda(t) \big(\|\nabla u(t,\cdot)\|_{L^2}+\|u_t(t,\cdot)\|_{L^2} \big)\le C. $$
\end{thm}
\begin{proof} The proof consists of two parts. 

{\sl Part 1.}  We denote 
$$\mathcal E_*(t,\xi)=\begin{cases}
M^{-1}  \lambda^{-1}(t_\xi)\exp\left(\i \int_{t_\xi}^t \mathcal D_m(\theta,\xi)\d\theta\right)M,\qquad &t\ge t_\xi,\\
\lambda^{-1}(t)I,\qquad &t\le t_\xi,\end{cases}$$ 
where we used for convenience the abbreviation $t_\xi=t_\xi^{(2)}$. In a first step we show that the limit
$$ \mathcal W(\xi)=\lim_{t\to\infty} \mathcal E_*^{-1}(t,\xi) \mathcal E(t,0,\xi) $$
exists uniformly on $|\xi|\ge c$ for any given $c>0$ and defines an invertible matrix $\mathcal W(\xi)$ for $\xi\ne0$. Indeed, the representation of Lemma~\ref{lem44} shows that
$$   M^{-1}\exp\left(-\i \int_{t_\xi}^t \mathcal D_m(\theta,\xi)\d\theta\right)M \mathcal E(t,t_\xi,\xi)
   \to M^{-1} \mathcal Q_m(\infty,t_\xi,\xi) \left(\prod_{k=0}^{m-1} N_k(t_\xi,\xi)\right)M,
$$
where the difference between left and right hand side can be estimated by $((1+t)|\xi|)^{-1}$, which tends uniformly to zero on any set $|\xi|\ge c$ with $c>0$. Note, that the right hand side is independent of $c$ and belongs to $L^\infty(\R^n)$ and $\mathcal W(\xi)$ is obtained after multiplication by $\mathcal E(t_\xi,0,\xi)$.

Therefore, $\mathcal W(\xi)\in L^\infty(\R^n)$ is well-defined and it remains to check the invertibility. For this we apply Liouville theorem to our initial system. This gives
$$ \det\mathcal E(t,s,\xi) = \exp\left(\i\int_s^t \mathrm{tr}\, A(\tau,\xi)\d\tau\right) =\exp\left(-2\int_s^t b(\tau)\d\tau\right) \approx \frac{\lambda^2(s)}{\lambda^2(t)}, $$
while 
$$\left| \det  \exp\left(-\i \int_{t_\xi}^t \mathcal D_m(\theta,\xi)\d\theta\right) \right|=\exp\left(2\int_s^t \tau_m(\theta,\xi)\d\theta\right) \approx \frac{\lambda^2(t)}{\lambda^2(s)}. $$
Therefore, $|\det\mathcal W(\xi)|\approx 1$ and thus $\mathcal W(\xi)$ and $\mathcal W^{-1}(\xi)$ are both uniformly bounded.

Note, that the matrix $\mathcal E_*(t,\xi)$ is a scalar multiple of a unitary matrix, the factor essentially given by $\lambda^{-1}(t)$. 

{\sl Part 2.} We consider the dense subspace $$L^2_{|\xi|\ge c} = \{ f\in L^2\,:\, \mathrm{dist}(0, \mathrm{supp}\,\hat f)\ge c\}\subseteq L^2(\R^n;\C^2).$$ 
In Part 1 we have shown that $\mathcal E_*^{-1}(t,\xi)\mathcal E(t,0,\xi)\to\mathcal W(\xi)$ uniformly  in $|\xi|\ge c$. Therefore, on the operator-level 
$$ \mathcal E_*^{-1}(t,\D) \mathcal E(t,0,\D) \to \mathcal W(\D) $$
pointwise on $L^2_{|\xi|\ge c}$. Theorem~1 provides us with the norm-estimate 
$$\|\mathcal E(t,0,\D) \diag(|\D|/\langle\D\rangle,1)\|_{L^2\to L^2}\lesssim 1,$$ such that by Banach-Steinhaus theorem the strong convergence
$$ \slim_{t\to\infty} \mathcal E_*^{-1}(t,\D) \mathcal E(t,0,\D) \diag(|\D|/\langle\D\rangle,1)= \mathcal W(\D) \diag(|\D|/\langle\D\rangle,1)$$
follows on $L^2(\R^n;\C^2)$. 
Therefore, for {\em all} $V_0\in\diag(|\D|/\langle\D\rangle,1) L^2(\R^n,\C^2)$
\begin{align*} &\| \mathcal E_*^{-1}(t,\D)\mathcal E(t,0,\D)  V_0 - 
\mathcal W(\D)  V_0 \|_{L^2} \\
\approx& \| \lambda(t) \mathcal E(t,0,\D)  V_0 - \lambda(t) \mathcal E_*(t,\D)\mathcal W(\D) V_0
\|_{L^2} \to 0,
\end{align*}
while we already know that $\|\lambda(t) \mathcal E_*(t,\D)\mathcal W(\D) V_0\|_{L^2}\approx 1$. Therefore,
the first term satisfies two-sided bounds $\| \mathcal E(t,0,\D)  V_0\|_{L^2}\approx \lambda^{-1}(t)$ and
the theorem is proven.
\end{proof}

\begin{appendix}
\section{Elementary properties of stabilising functions} 
We collect some properties of stabilising functions. For $f\in L^1_{loc}(\R_+)$ and
$\alpha\in\R$ we write $f\rightsquigarrow\alpha$ if $\int_0^t |f(s)-\alpha|\d s = o(t)$. 

1. $\alpha$ is uniquely determined. Assume the condition is also satisfied for $\alpha'$, then
$$ t|\alpha-\alpha'| = \int_0^t |\alpha-\alpha'|\d s \le \int_0^t |f(s)-\alpha|\d s+\int_0^t |f(s)-\alpha'|\d s = o(t) $$
and the assertion follows.

2. The number $\alpha$ does not depend on the lower bound of the integral, the number $0$ is used for convenience only.

3. If $f_1\rightsquigarrow\alpha_1$ and $f_2\rightsquigarrow\alpha_2$ then
$f_1+cf_2\rightsquigarrow \alpha_1+c\alpha_2$.

4.  Stabilisation $f\rightsquigarrow\alpha$ does not imply convergence of $f(t)$, but if we additionally know that the limit exists, then it must be equal to $\alpha$.

5. If $f\rightsquigarrow \alpha$ and $g$ is monotone with $g'(t)$ bounded, then $f\circ g\rightsquigarrow \alpha$. This follows by substitution in the integral,
$$ \int_{g^{-1}(0)}^{g^{-1}(t)} |f(g(s'))-\alpha|\d s' =  \int_0^t |f(s)-\alpha| \d s= o(t) = o(g^{-1}(t)). $$

6. On the contrary, if $f\rightsquigarrow\alpha$ and $g$ is Lipschitz, then $g\circ f\rightsquigarrow g(\alpha)$. Indeed, the Lipschitz condition implies directly
$$ \int_0^t |g(f(s))-g(\alpha)|\d s\le L \int_0^t |f(s)-\alpha|\d s = o(t). $$
Note that it is sufficient to require that $g$ is locally Lipschitz if $f$ is bounded, i.e. we need the Lipschitz condition on the range of $f$. 
\end{appendix}

{\bf Acknowledgements.} These notes are based on discussions on the subject held when the second author visited the first one in April 2007. This research trip was supported by the German science foundation (DFG) as project 446 JAP 111/3/06.


\begin{thebibliography}{99}
\bibitem{Emmerling07} J. Emmerling.
\newblock Wave equations with time-dependent coefficients.
\newblock {\em Adv. Math. Sci. Appl.}, to appear.
\bibitem{Hirosawa06} F. Hirosawa.
\newblock On the asymptotic behavior of the energy for wave equations
with time-depending coefficients.
\newblock {\em Math. Ann.},  339(4):819--839, 2007.
\bibitem{HW} F. Hirosawa, J. Wirth.
\newblock  Generalized energy conservation law for the wave equations with variable propagation speed.
\newblock in preparation.
\bibitem{Reissig:1999}
M.~Reissig and K.~Yagdjian.
\newblock {One application of Floquet's theory to $L_p$-$L_q$ estimates for
  hyperbolic equations with very fast oscillations.}
\newblock {\em Math. Methods Appl. Sci.}, 22(11):937--951, 1999.
\bibitem{Yagdjian:2000a}
M.~Reissig and K.~Yagdjian.
\newblock {About the influence of oscillations on Strichartz-type decay
  estimates.}
\newblock {\em Rend. Semin. Mat., Torino}, 58(3):375--388, 2000.
\bibitem{Reissig:2005}
M.~Reissig and J.~Smith.
\newblock {$L^p$-$L^q$ estimate for wave equation with bounded time dependent
  coefficient.}
\newblock {\em Hokkaido Math. J.}, 34(3):541--586, 2005.
\bibitem{Wirth04} J. Wirth.
\newblock Solution representations for a wave equation with weak dissipation.
\newblock{\em Math. Meth. Appl. Sci.}, 27:101--124, 2004.
\bibitem{Wirth05} J. Wirth.
\newblock {\em Asymptotic properties of solutions to wave equations with time-dependent dissipation}.
\newblock PhD thesis, TU Bergakademie Freiberg, 2005.
\bibitem{Wirth06} J. Wirth.
\newblock {Wave equations with time-dependent dissipation I. Non-effective dissipation.}
\newblock{\em J. Differential Equations}, 222:487--514, 2006.
\bibitem{Wirth:2007}
J.~Wirth.
\newblock {Wave equations with time-dependent dissipation. II: Effective
  dissipation.}
\newblock {\em J. Differential Equations}, 232(1):74--103, 2007.
\bibitem{Wirth:2007a} 
J.~Wirth.
\newblock Scattering and modified scattering for abstract wave equations with
  time-dependent dissipation.
\newblock {\em Adv. Differential Equations}, 12(10):1115--1133, 2007.
\bibitem{Yagdjian:2001}
K.~Yagdjian.
\newblock {Parameteric resonance and nonexistence of the global solution to
  nonlinear wave equations.}
\newblock {\em J. Math. Anal. Appl.}, 260(1):251--268, 2001.
\end{thebibliography}
\end{document}